\documentclass{amsart}
\usepackage{amsmath}
\usepackage{amssymb}
\usepackage{mathtools}
\usepackage{amsthm}
\usepackage{hyperref}
\newtheorem{problem}{Problem}

\newtheorem{theorem}{Theorem}

\newtheorem{corollary}{Corollary}

\newtheorem{lemma}{Lemma}

\title{Arithmetic Progressions in Midpoint Colourings}
\author{Tomasz Kościuszko}
\begin{document}
\begin{abstract}
We introduce an asymmetric variant of a classic problem by Roth about colourings of integers and midpoints. The new problem cannot be solved with the standard approach by Erdös-Sárközy-Sós which uses symmetry. We describe a Fourier Analytic approach which gives asymptotically tight bounds. In the $\mathbb{F}_3^n$ setting, the density increment we show is efficient and together with the Freiman-Ruzsa Theorem provides an affine subspace of near optimal codimension within the distinguished set. This approach adopted to the setting of the integers from $1$ to $N$ via Bohr sets and Bogolyubov-Ruzsa Lemma gives a progression of length being a power of $N$ which only depends on the number of colours. We then use Chang's Lemma and Balog-Szemerédi-Gowers Theorem to further improve the dependence on the number of colours.
\end{abstract}
\maketitle
\section{The motivation for uncoloured midpoints}
Roth stated the following problem, see~\cite{ess}.
\begin{problem}
    Let $k$ be a fixed positive number and let $1,2,\cdots,N$ be $k$-coloured. Let $R\subseteq [N]$  and assume that whenever for distinct $x$ and $y$ there is $r\in R$, such that $x+y=2r$, then $x$ and $y$ have different colours. How large can $R$ be in terms of $N$?
\end{problem}
Erdős-Sárközy-Sós~\cite{ess} answered it by proving $|R|=o_k(N)$, specifically they showed that $|R|\leq N^{1-2^{-k-1}}.$ They proved it by a purely combinatorial elementary argument (they show that otherwise $R$ contains a translate of a $k$-dimensional hilbert cube, which leads to a contradiction).
The most desired result would be to improve this bound by using Fourier Analysis. It is possible do that for $\mathbb{F}_3^n$ instead of the integers which we describe in Section 2. Unfortunately, this is not the best possible approach as we can use the cap set theorem of Ellenberg and Gijswijt~\cite{ellenberg-gijswijt} to get even a stronger bound. We mention Fourier Analysis here as it seems that a more efficient execution of the density increment than the one we explore could actually improve the bounds.

Another similar problem is a variation of this question asked by Green in his list of \href{https://people.maths.ox.ac.uk/greenbj/papers/open-problems.pdf}{100 open problems, problem 25}~\cite{green}.
\begin{problem}
    Let $N$ be a positive integer and let $A_1\cup A_2\cup\cdots\cup A_k=[N]$ be a partitioning. Assume there exists a set $R\subseteq [N]$ with $|R|\geq N/10$ such that for every $i$ we have
    $$(A_i\hat{+}A_i)\cap (2\cdot R)=\emptyset$$
    What lower bound on $k$ can we prove in terms of $N$?
\end{problem}
The result by Erdős-Sárközy-Sós immediately implies $k\gg\log\log N$. There is however a much stronger lower bound coming from the Kelley--Meka bound for Roth's Theorem on arithmetic progressions (it is folklore knowledge and it is the same reasoning as the application of the cap set theorem we mentioned above).
\begin{theorem}(Kelley-Meka~\cite{kelley})
    Let $A\subseteq 1,2,...N$ be a set free of 3-term arithmetic progressions. Then $|A|\leq \exp(-\log^{c} N)N$.
\end{theorem}
Notice that if in Problem~2 we take the colour with the largest intersection with $R$ (say $A_i$) and define $R'=R\cap A_i$ then $R'$ must be free of 3-term arithmetic progressions. Thus $N/10k\leq|R|/k\leq|R'|\leq \exp(-\log^{c} N)N$ and we get $k \geq \exp(\log^{c} N)/10$ which gives a much stronger lower-bound.

The same reasoning can be used in $\mathbb{F}_3^n$ to show $|R|=O(2.756^{n})$.

This way of solving the question of Roth seems somewhat unsatisfactory. It relies solely on the fact that $R$ contains a structure that is impossible to colour (it is the same with hilbert cube approach), everything outside of $R$ is disregarded. We propose the following problem which asks roughly the same question but leaves $R$ entirely uncoloured. This seems to force the question that Roth might have had in mind when asking about Problem 1.
\begin{problem}\label{problem:integers_asymmetric}
    Let $N$ be a positive integer and let $A_1\cup A_2\cup\cdots\cup A_k\cup R=[N]$ be a partitioning. Assume that if $x,y\in A_i$ for some $A_i$, then $(x+y)/2\notin R$ and that $|R|\geq N/10.$ What is the longest arithmetic progression contained in $R$?
\end{problem}
The reason we ask for longest arithmetic progression and not to lower bound $k$ is the following trivial construction: let $A_1,A_2,\cdots, A_k$ all be intervals with total length at most $9N/10$. Then for every $x,y\in A_i$ if $(x+y)/2$ is an integer then it is contained in $A_i$ so it is outside $R$.

The same problem can be asked in the finite field setting $\mathbb{F}_3^n$.
\begin{problem}\label{problem:fields_asymmetric}
    Let n be a positive integer and let $A_1\cup A_2\cup\cdots\cup A_k\cup R=\mathbb{F}_3^n$ be a partitioning. Assume that if $x,y\in A_i$ for some $A_i$, then $(x+y)/2\notin R$ and that $|R|\geq 3^n/10.$ What is the largest translate of a subspace contained in $R$?
\end{problem}

In section 2 we show that a translate of a subspace of dimension at least $n-O(k^2)$ can be found inside $R$ by the density increment method based on a large fourier coefficient and applying the Freiman-Ruzsa's Theorem. On the other hand we can easily construct a colouring where $R$ does not contain a translate of a subspace of dimension larger than $n-k$. 
To do that let us randomly choose $A_1,A_2,\cdots, A_k$ to be affine subspaces of codimension $2\log_3(k)$ and $R$ to be the complement. It is not hard to see that on average $R$ will contain an affine subspace of dimension at most $n-k$.

In section 3 we show equivalent reasoning in the integers and point out the limitation to the iteration strategy. By iterating on arithmetic progressions we show is that $R$ contains an arithmetic progression of length $N^{\exp(-O(k\log (2k)))}$.

Here the gap between the lower bound and the upper bound we can prove is significant - by taking $A_i$'s to be multiples of primes of order $N^{1/(k+1)}$ we can construct $R$ without progressions longer than $N^{1/(k+1)}$. 

In section 4 we prove a more efficient version of density increment lemma which uses a range of tools: Chang's Lemma~\cite{chang}, Plünnecke's inequality~\cite{plunnecke}, Bogolyubov-Ruzsa Lemma~\cite{freiman, sanders} and Balog-Szemerédi-Gowers Theorem~\cite{balog-szemeredi, gowers1998}. As a consequence we show that $R$ contains a progression of length $N^{\exp(-O(k\log\log(2k)))}$. This means the correct extremal value is $N^{c_k}$, where $c_k$ depends only on $k$ and
$$-k\log\log(2k)\ll \log c_k\leq -\log(k+1).$$

\section{An argument in the model case $\mathbb{F}_3^n$}
For the sake of presentation we start with the case $\mathbb{F}_3^n$ in which the method works smoothly. This argument can be used to show $|R|\leq 3^{n-n/4k}$ in Problem~1, yet again because cap set bounds are very strong in $\mathbb{F}_3^n$ the approach described in the previous section where we directly use the 3-AP free bounds is superior. Thus we focus on Problem~3 where that method does not apply.

In the whole document we use well-known results in Discrete Fourier Analysis and the theory of Bohr sets. For the standard identities and theorem such as convolution identity, fourier inversion and Parseval's Theorem we refer to \cite[Chapter~4]{tao-vu}.
\begin{lemma}\label{lem:fields_increment}
Let $R\subseteq \mathbb{F}_3^n$ and $A\subseteq \mathbb{F}_3^n$ with $|A|\geq 3^n\beta$. Suppose that there are no solutions to the equation $x+y=2r$ where $x,y\in A$ and $r\in R$. Then there exists a subspace $V$ of codimension $1$ such that for some translate $t$ we have $$|(t+V)\cap R|/|V|\geq \Bigl(1+\frac{\beta}{2}\Bigr)|R|/3^n.$$
\end{lemma}
\begin{proof}
Let us define Fourier Transform of $f:\mathbb{F}_3^n\rightarrow \mathbb{C}$ as 
$$\hat{f}(\gamma):=\mathop{\mathbb{E}}_x f(x)\overline{\gamma(x)},$$
where $\gamma$ is an element of the dual group of $\mathbb{F}_3^n$.\\
From our assumption we have
$$\sum_x 1_{2\cdot R}(x)1_A*1_A(x) = 0,$$
which by Plancherel's Theorem can be written as
$$\sum_{\gamma}\widehat{1_{2\cdot R}}(\gamma)\overline{\widehat{1_A}(\gamma)^2}=0.$$
We separate the trivial term $\gamma=1$ to get
$$\frac{1}{3^{3n}}|2\cdot R||A|^2 = 0 -\sum_{\gamma\neq 1}\widehat{1_{2\cdot R}}(\gamma)\overline{\widehat{1_A}(\gamma)^2}.$$
Write $M:=\mathop{\mathrm{max}}_{\gamma\neq 1}|\widehat{1_{2\cdot R}}(\gamma)|$, then we can estimate
$$\frac{1}{3^{3n}}|2\cdot R||A|^2\leq M\sum_{\gamma\neq 1}|\widehat{1_A}(\gamma)|^2 \leq \frac{M}{3^{n}}|A|.$$
Rearranging for $M$ and estimating $|A|\geq 3^n\beta$ we get
$$M\geq \frac{|2\cdot R|\beta}{3^n}.$$
Let $\gamma_{0} \neq 1$ be the character achieving the maximum $M$ and define
$$V = \{x\in \mathbb{F}_3^n : \gamma_0(x)=1\}.$$ Since $\gamma_{0}$ is non-trivial, $V$ is a subspace of codimension 1. Let $\alpha_0,\alpha_1,\alpha_2$ be densities of $2\cdot R$ within the three cosets of $V$. On these cosets $\gamma_0$ takes values $1,\omega,\omega^2$ respectively, where $\omega = \exp(2\pi i /3)$. We can then write
$$\frac{1}{3} \left| \alpha_0 + \alpha_1 \bar{\omega} + \alpha_2 \bar{\omega}^2 \right| = |\widehat{1_{2\cdot R}}(\gamma_{0})|$$
Since $1+\bar{\omega}+\bar{\omega}^2=0$ we can write
$$|(\alpha_0-\alpha)+(\alpha_1-\alpha)\bar{\omega}+(\alpha_2-\alpha)\bar{\omega}^2|= 3M.$$
Thus for some $i\in\{0,1,2\}$ we have
$$\alpha_i-\alpha \geq \frac{|2\cdot R|\beta/2}{3^n},$$
which is the desired density increment of $2\cdot R$ on a translate of $V$. Notice that a factor of $1/2$ has to be present because we are only interested in density increments and not decrements. If $2r\in t+V$ for some $r\in R$ then $r\in t/2+V$ so we automatically get the same increment for $R$ on a translate of the same subspace.
\end{proof}
To prove existence of large affine subspace in $R$ we will iterate Lemma~\ref{lem:fields_increment} and after the last iteration use the following version of the Freiman-Ruzsa Theorem coming from~\cite{even-zohar-lovett}. It is likely that by using the Bogolyubov-Ruzsa lemma as in the proof of Lemma~\ref{lem:termination} we could improve the bound $O(k^2)$, however it would much complicate the proof and the priority of this section is clear presentation.
\begin{theorem}[Even-Zohar, Lovett]\label{thm:freiman_ruzsa}
Let $A\subseteq \mathbb{F}_3^n$ and $|A+A|\leq K|A|.$ Then $A$ is contained in an affine subspace of size at most $3^{C_{FR}K}|A|$.
\end{theorem}
\begin{theorem}\label{thm:f3_subspace}
Let $k,n$ be positive integers and let $A_1\cup A_2\cup\cdots\cup A_k\cup R=\mathbb{F}_3^n$ be a partitioning. Assume that if $x,y\in A_i$ for some $A_i$, then $(x+y)/2\notin R$ and that $|R|\geq 3^n/10.$ Then $R$ contains an affine subspace of codimension at most $O(k^2).$
\end{theorem}
\begin{proof}
    We will iteratively apply Lemma~\ref{lem:fields_increment} $z$ times ($z$ to be chosen later in the proof). In the first iteration we let $V_1=\mathbb{F}_3^n$ and $R_1=R$. $A^1$ will be the largest colour present in $V_1$ (we use superscript not to confuse with the colours $A_i$ from the partition). As a result of the lemma for $V_i,R_i,A^i$ we get $t_{i+1}+V_{i+1}$ on which $R_i$ has increased density. Because the equation $x+y=2r$ is translation invariant the same rule holds, that is if $A^{i+1}$ is the largest colour in $t_{i+1}+V_{i+1}$ and $x,y\in t_{i+1}+A^{i+1}$ we have $(x+y)/2\notin t_{i+1}+R_{i+1}$. Thus further iteration is possible. 

    Let $\alpha_i=|R_i|/|V_i|$ so that $\alpha_1\geq 1/10$, then Lemma~\ref{lem:fields_increment} guarantees
    $$\alpha_{i+1}\geq \Bigl(1+\frac{1-\alpha_i}{2k}\Bigr)\alpha_i.$$
    Clearly there exists an absolute constant $C$, such that $1-\alpha_i$ halves every $Ck$ steps, therefore
    $$1-\alpha_{z}\leq 2^{-\lfloor z/Ck \rfloor}.$$
    Choose $z = 4\lceil C_{FR}\rceil \lceil C\rceil k^2=O(k^2)$ so that $1-\alpha_z < 3^{-2C_{FR}k}.$
    
    We now remove all colours one by one from $V_z$ with $k$ applications of Theorem~\ref{thm:freiman_ruzsa}. Choose the largest colour in $V_z$, call it $A$. Because of our assumption $(A+A)\cap(2\cdot R)=\emptyset$ and so $|A+A|\leq |V_z|-|R|\leq k|A|$. That means $A$ is contained in an affine subspace of size at most $3^{C_{FR}k}|A|$. Because $3^{C_{FR}k}|A|<|V_z|$ this is a proper affine subspace. Its two parallel affine subspaces contain no points from $A$. On one of them density of the remaining colours is at most $\frac{3}{2}2^{-s}$. We repeat this argument $k$ times to obtain an affine subspace of codimension at most $z+k=O(k^2)$ contained entirely inside $R_z$.

    We always have a proper affine subspace containing any colour as
    $$3^{C_{FR}k}\Bigl(\frac{3}{2}\Bigr)^k2^{-s} < 1.$$
\end{proof}
\section{Limitations of the argument in the integers}
The iterative argument in the integers is similar. If $A,R\subseteq \mathbb{Z}_N$ we can translate Lemma~\ref{lem:fields_increment} without difficulty and find a large fourier coefficient of $1_{2\cdot R}$.

The purpose of this section is pointing out why iterative approach which worked for $\mathbb{F}_3^n$ is suboptimal in terms of $k$ when considering the case of the integers. To do so we prove a result corresponding to Lemma~\ref{lem:fields_increment} but assuming $A,R$ are subsets of a Bohr set with rank $d$ and width $\rho$. As it turns out, the increment we can obtain decreases exponentially with growing $d$. For that reason the lemma is basically only meaningful when $d$ is constant.

We can still recover a meaningful bound for Problem~3 by choosing an arithmetic progression, or alternatively a constant rank Bohr set on every step. This causes dramatic drop of the size of our ground set from $N$ to $N^{1/\exp(s)}$ after $s$ steps. It would be very interesting to see a way to navigate around this difficulty and possibly address Problems~1 and~2.

\begin{lemma}\label{lem:bohr_set_increment}
Let $N\geq 1$ be an odd integer and let $A,R\subseteq B$, where $B\subseteq \mathbb{Z}_N$ is a regular Bohr set of rank $d$ and width $\rho$. Let $|R|=\alpha|B|$ and $|A|= \beta |B|$. Suppose that there are no solutions to the equation $x+y=2r$ where $x,y\in A$ and $r\in R$.
Then there exists a Bohr set $B'$ of rank $d+1$ and width at least $\Omega(\rho\beta/2^{d+1})$ such that we have
$$|(t+B')\cap R|/|B'|\geq \Bigl(1+\frac{\beta}{2^{d+4}}\Bigr)\alpha.$$
\end{lemma}
\begin{proof}
We begin by choosing a possibly large subset $A'\subseteq A$ with the property $A'+A'\subseteq 2\cdot B$. It is not hard to see that by choosing only elements with the same parity after multiplication by all characters generating $B$ we get $|A'|\geq |A|/2^d$.
Consider the function $f(x)=1_{2\cdot R}(x)-\alpha 1_{2\cdot B}(x)$. Define its fourier coefficients as 
$$\hat{f}(\gamma):=\mathop{\mathbb{E}}_x f(x)\overline{\gamma(x)},$$
where $\gamma$ is an element of the dual group of $\mathbb{Z}_N$. Notice that $\hat{f}(\gamma_0)=0$, where $\gamma_0=1$ is the trivial character. Then by Parseval's Theorem we have
$$\sum_{x}f(x)1_{A'}*1_{A'}(x) = N^2\sum_{\gamma}\hat{f}(\gamma)\overline{\hat{1_{A'}}(\gamma)}^2.$$
as we use unnormalised convolution. By using the assumption we have
$$\alpha|A'|^2=\alpha\sum_{x\in {2\cdot B}}1_{A'}*1_{A'}(x)= \Bigl|\sum_x f(x)1_{A'}*1_{A'}(x)\Bigr|\leq N^2\sum_{\gamma}|\hat{f}(\gamma)||\hat{1_{A'}}(\gamma)|^2.$$
Notice that if we were considering $A$ instead of $A'$, then the first equality would not be true. In this sense the loss of the $2^d$ factor is necessary. As a consequence we get
$$\frac{\beta |A'||R|}{2^dN^2} \leq \frac{\alpha|A'|^2}{N^2} = \sum_{\gamma}|\hat{f}(\gamma)||\hat{1_{A'}}(\gamma)|^2 \leq |\hat{f}(\gamma_{\max})|\sum_{\gamma}|\hat{1_{A'}}(\gamma)|^2 = |\hat{f}(\gamma_{\max})|\frac{|A'|}{N}$$
and thus $|\hat{f}(\gamma_{\max})|\geq \frac{\beta}{2^d}\cdot\frac{|R|}{N}$. 

To obtain a density increment, choose $B'$ to be generated by the same characters as $B$ with the addition of $\gamma_{\max}$. Let the width be $\rho'\gg \rho\beta/2^{d+1}$ so that $B'+B\subseteq B_{1+\rho'/\rho}$ and $|B_{1+\rho'/\rho}|\leq (1+\beta/2^{d+4})|B|$ because of regularity of $B$. By considering $\sum_t 1_B*1_{B'}(t)f(t)$ in a standard way we find a density increment on a translate of $B'$.
\end{proof}

Since we cannot allow the $2^{-d}$ factor we will use the following classic lemma (can be accessed in~\cite[Exercise~4.4.5]{tao-vu}) to keep the rank of the Bohr set at $1$, that is to always be in the situation where the Bohr set is an interval in the group $\mathbb{Z}_M$ for an appropriate $M$.

\begin{lemma}\label{lem:find_progression}
Let $N\geq 1$ be an integer and let $B\subseteq\mathbb{Z}_N$ be a Bohr set of rank $d$ and width $\rho\gg1/d^2$. Suppose that $d^{100d}\leq N$, then $B$ contains an arithmetic progression of length at least $N^{1/2d}$.
\end{lemma}

Let us first prove a corollary which only deals with arithmetic progressions and not Bohr sets.
\begin{corollary}\label{cor:progressions}
Let $N\geq 1$ be an integer and let $A,R\subseteq [N]$, with $|A|\geq \beta N\geq N^{4/5}$. Suppose that there are no solutions to the equation $x+y=2r$ where $x,y\in A$ and $r\in R$.
Then there exists an arithmetic progression $P$ with length at least $N^{1/4}$ such that
$$|P\cap R|/|P|\geq \Bigl(1+\beta/32\Bigr)\frac{|R|}{N}.$$    
\end{corollary}
\begin{proof}
Consider $A,R$ as subsets of $\mathbb{Z}_{2N+1}$. Clearly there are no solutions to the equation $x+y=2r$ with $x,y\in A$ and $r\in R$ because there is no wrap around modulo $2N+1$. Because $[N]$ can be viewed as a translate of rank 1 Bohr set (with width $1/4$) by Lemma~\ref{lem:bohr_set_increment} there exists a Bohr set $B'$ with rank 2 and width $\Omega(\beta)$ such that for some translate $t$ of $B'$ we have  
$$|(t+B')\cap R|/|B'|\geq \Bigl(1+\beta/32\Bigr)\frac{|R|}{N}.$$ 
By Lemma~\ref{lem:find_progression} we can find $P$ contained in a dilate of $B'$ of length at least $N^{1/4}$ and with simple averaging argument we can ensure the same or larger density of $R$ on this progression.

\end{proof}

Now by iterating Corollary~\ref{cor:progressions} we are able to show that $R$ contains an arithmetic progression of length at least $N^{\exp(-O(k\log (2k)))}$.
This is still much larger progression that the one guaranteed on the basis of density $1/10$ alone by Szemerédi's Theorem as proved by Gowers~\cite{gowers2001} which is about $\Omega(\log\log\log\log\log N)$.
We stop iterating once $R$ has large enough density for us to apply the following lemma successfully.
\begin{lemma}\label{lem:termination}
Let $k,N$ be positive integers with $k<N$ and let
$$A_1\cup A_2\cup\cdots\cup A_k\cup R=[N]$$ 
be a partitioning. Assume that if $x,y\in A_i$ for some $A_i$, then $(x+y)/2\notin R$ and that $|R|\geq (1-(10k)^{-6})N.$ Then $R$ contains an arithmetic progression of length at least $N^{\exp(-O(k\log\log (2k)))}$.  
\end{lemma}
Before we give the proof let us state the Bogolyubov-Ruzsa Lemma with Generalised Arithmetic Progressions. A Generalised Arithmetic Progression $\mathcal{P}$ of dimension $D$ is defined by generators $d_1,d_2,\dots,d_D$ and lengths $L_1,L_2,\dots,L_D$ in the following way.
$$\mathcal{P}=\{i_1d_1+i_2d_2+\dots i_Dd_D : 0\leq i_j<L_j\}.$$
Now we can state the Bogolyubov-Ruzsa Lemma (in the version by Sanders~\cite{sanders-structure}).
\begin{theorem}[Bogolyubov-Ruzsa Lemma]
Let $A\subseteq [N]$ with $|A+A|\leq K|A|.$ There exists a Generalised Arithmetic Progression with dimension at most $O(\log^4 K)$ and size at least $\exp(-O(\log^4 K))|A|$, such that
$$\mathcal{P}\subseteq 2A-2A.$$
\end{theorem}
\begin{proof}(Of Lemma~\ref{lem:termination})
We will show how to find a progression of length $N^{1/O(\log^4k)}$ which avoids one of the colours and on which density of $R$ is at least $1-(k-1)^{-6}$. Then the lemma is proved by induction on $k$ as with $k=0$ the statement is vacuously true.

Let $I_0=[N]$ and iterate up to $1000k$ times the following:
\begin{itemize}
    \item Let $\rho_i = 1-|R\cap I_i|/|I_i|$ be the compliment (of $R$) density within $I_i$.
    \item {If $i=1000k$ terminate with $I=I_i$ (CASE 1). In that case $$\rho=\rho_i<(1-1/200k)^{1000k}\rho_0\leq \rho_0/4.$$}
    \item Split our interval $I_i=\{1,2,3,\dots, |I_i|\}$ into even numbers and odd numbers.
    \item If on both parts compliment density is at least $(1-1/200k)\rho_i$ terminate with $I=I_i$ (CASE 2).
    \item Otherwise continue iteration with $I_{i+1}$ being the parity class with complement density at most $(1-1/200k)\rho_i$.
\end{itemize}

Let $M=|I|$ and choose the largest colour within $I$, call it $A$. Again divide $I$ into two parity classes and choose one where $A$ has at least $\lfloor \rho M/2k\rfloor$ elements, call it $J$ and the other one $J'$. Let $B=A\cap J$ and $\rho_J=1-|J\cap R|/|J|.$

Crucially because $(B+B)\cap (2\cdot R)=\emptyset$ we have
$$|B+B|\leq M-|I\cap (2\cdot R)|\leq k|A|\leq 2k|B|.$$
Thus $B$ is a set of small doubling. If $k|B|<\sqrt{M}$, then $R$ trivially contains an arithmetic progression of length at least $\sqrt{M}$ by the pigeonhole principle, so let us assume the opposite. By the Bogolyubov-Ruzsa's Lemma we find a Generalised Arithmetic Progression $\mathcal{P}$ with dimension at most $O(\log^4 k)$ and size at least $\exp(-O(\log^4k))|B|$ such that
$$\mathcal{P}\subseteq 2B-2B.$$
Choose $P$ to be the longest dimension of $\mathcal{P}$ so that 
$$|P|\geq |\mathcal{P}|^{1/O(\log^4 k)}\geq M^{1/O(\log^4 k)}.$$
Define $P'$ to be the first $\lfloor|P|/10^7k\rfloor$ elements of $P$.

By containment and the Plünnecke's inequality we have
$$|B-P'|\leq |2B-3B|\leq (2k)^5|B|.$$
And so there are at most $(2k)^5 \rho M$ translates of $P'$ which contain any element of $B$ and thus $A$ because elements of $A\setminus B$ trivially have the other parity.

Thus the number of translates of $P'$ which are contained wholly in $I\setminus A$ is at least
$$|J|-2M/10^6k-(2k)^5\rho M\geq |J|(1-1/200k).$$
Averaging their complement densities we obtain one translate $t+P\subseteq I\setminus A$ such that its complement density is
$$1-|(t+P')\cap R|/|P'| \leq \frac{\rho_J}{1-1/200k}.$$

We now finish the proof regardless of whether we are in CASE~1 or CASE~2. In both cases we choose $t+P'$ which has size at least $(N/2^{1000k})^{1/O(\log^4 k)}=N^{1/O(\log^4 k)}$.

If we are in CASE~1, then $\rho_J\leq 2\rho\leq \rho_0/2$ and so
$$\frac{\rho_J}{1-1/200k}\leq \rho_0\leq (10k)^{-6}$$
as desired.

If we are in CASE~2, then $\rho_J\leq (1+1/100k)\rho$ and so
$$\frac{\rho_J}{1-1/200k}\leq (1+1/25k)\rho_0\leq (1+1/25k)(10k)^{-6}\leq 10^{-6}(k-1)^{-6}$$
as desired.
\end{proof}
\begin{theorem}\label{thm:progression_in_integers}
Let $k,N$ be positive integers with $k<N$ and let
$$A_1\cup A_2\cup\cdots\cup A_k\cup R=[N]$$ 
be a partitioning. Assume that if $x,y\in A_i$ for some $A_i$, then $(x+y)/2\notin R$ and that $|R|\geq N/10.$ Then $R$ contains an arithmetic progression of length at least $N^{\exp(-O(k\log(2k)))}$.    
\end{theorem}
\begin{proof}
    We apply Corollary~\ref{cor:progressions} $s$ times (to be decided later in the proof). Let $N_0=N$ and denote the length of the progression in the first application as $N_1$, second $N_2$ etc. Let $\alpha_0=\alpha$ and denote the density of $R$ in the first application as $\alpha_1$, second $\alpha_2$ etc. Denote the subsequent progressions as $P_0=[N], P_1,\cdots, P_s$. Let $A^{i}$ be the largest intersection of any colour with $P_i$. Because there at $k$ colours we have $|A_i|\geq \frac{1-\alpha_i}{k}N_i$.
    
    From the conclusion of Corollary~\ref{cor:progressions} we know that for any $i=0,1,\cdots,s-1$ we have
    $$N_{i+1}\geq N_i^{1/4}$$
    and
    $$\alpha_{i+1}\geq \Bigl(1+\frac{1-\alpha_i}{2k}\Bigr)\alpha_i.$$
    Observe that
    $$\frac{1-\alpha_{i+1}}{1-\alpha_i} \leq 1-\frac{\alpha_i}{2k}\leq 1-\frac{1}{20k}$$
    and so we have $1-\alpha_{i+{20k}}\leq (1-\alpha_i)/2$. In other words, the proportion of the complement of $R$ in $P_i$ halves every $20k$ steps.

    In order to apply Lemma~\ref{lem:termination} the complement proportion must be at most $(10k)^{-6}$ as so $s=\lceil120k\log_2k\rceil=O(k\log(2k))$ suffices. Finally for the interval $N_s$ we apply Lemma~\ref{lem:termination} to obtain an arithmetic progression of length $N^{\exp(-O(k\log(2k)))}$ inside $R$.
    
\end{proof}
\section{Improved dependence on $k$}
The main result in the previous section shows the existence on an arithmetic progression of length at least $N^{\exp(-O(k\log(2k)))}$ in $R$. In this section we use additional observations to improve the negative exponent from $O(k\log(2k))$ to $O(k\log\log (2k))$.

This is accomplished by proving a more efficient version of the density increment Lemma~\ref{lem:bohr_set_increment} (and Corollary~\ref{cor:progressions}).
We split the argument into two cases - small additive energy of any large colour or big additive energy for all large colours.

Let us begin with a case of small additive energy. One of our tools is Chang's Lemma~\cite{chang} which we state here.
\begin{lemma}{(Chang)}\label{lem:chang}
Let $X\subseteq \mathbb{Z}_N$ and let 
$$\Lambda = \{\gamma:|\hat{1_X}(\gamma)|\geq \lambda|X|/N\}.$$ Then there is a set of characters $\Gamma$, such that $|\Gamma|\leq 2\lambda^{-2}\log(N/|X|)$ and $\Lambda\subseteq \mathrm{Span}(\Gamma)$.
\end{lemma}
\begin{lemma}\label{lem:small_energy}
Let $N\geq 1$ be an integer and let $A,R\subseteq \mathbb{Z}_N$. Let $|R|=\alpha N$ and $|A|=\beta N$ with $\alpha\geq 1/10$  and $\alpha+\beta\leq 1$. Assume that $E(A)\leq  |A|^3/M$ and that if $x,y\in A$ then $x+y\notin R$.
Then there exists an arithmetic progression $P$ of length at least $$N^{\Omega((\beta/(1-\alpha))^2\log(2/(1-\alpha))^{-1})}$$ such that we have
$$|P\cap R|/|P|\geq \Bigl(1+\Omega(\beta M)\Bigr)\frac{|R|}{N}.$$
\end{lemma}
\begin{proof}

Consider the balanced function 
$$f(x)=1_R(x)-\alpha$$
and define its fourier coefficients as 
$$\hat{f}(\gamma):=\mathop{\mathbb{E}}_x f(x)\overline{\gamma(x)},$$
where $\gamma$ is an element of the dual group of $\mathbb{Z}_{N}$. Then by Parseval's Theorem we have
$$\sum_{x}f(x)1_A*1_A(x) = N^2\sum_{\gamma}\hat{f}(\gamma)\overline{\hat{1_A}(\gamma)}^2.$$

By using the assumption we have
$$\alpha\beta^2 N^2=\alpha |A|^2 \leq \Bigl|\sum_x (1_{R}(x)-\alpha)1_A*1_A(x)\Bigr|\leq N^2\sum_{\gamma}|\hat{f}(\gamma)||\hat{1_A}(\gamma)|^2.$$
Let $\eta = \frac{\alpha\beta}{2}$ and define the large spectrum of $\hat{f}$ as
$$S := \{\gamma : |\hat{f}(\gamma)|\geq \eta\}.$$
Let us split the sum to the sum of elements in $S$ and other elements.
$$\alpha\beta^2\leq \sum_{\gamma\notin S}|\hat{f}(\gamma)||\hat{1_A}(\gamma)|^2 + \sum_{\gamma\in S}|\hat{f}(\gamma)||\hat{1_A}(\gamma)|^2\leq \eta\frac{|A|}{N} + \sum_{\gamma\in S}|\hat{f}(\gamma)||\hat{1_A}(\gamma)|^2$$
and so
$$\frac{\alpha\beta^2}{2} \leq \alpha\beta^2 - \frac{\eta\beta}{2}\leq \sum_{\gamma\in S}|\hat{f}(\gamma)||\hat{1_A}(\gamma)|^2.$$
Since $\sum_{x\in\mathbb{Z}_{N}}1_A*1_A(x)^2=E(A)\leq |A|^3/M$ we get by the Cauchy-Schwarz inequality
$$\frac{\alpha^2\beta^4}{4} \leq \Bigl(\sum_{\gamma\in S}|\hat{f}(\gamma)|^2\Bigr)\cdot\frac{8|A|^3}{MN^3}$$
and so
$$\frac{\alpha^2\beta M}{32} \leq \sum_{\gamma\in S}|\hat{f}(\gamma)|^2.$$

We now routinely obtain a density increment based on the large $L_2$ norm we just showed.

Define a regular Bohr set $B'$ of rank $|S|$ and width between $1/100$ and  $1/200$ generated by the characters from $S$. Then because $|\hat{\mu}_{B'}(\gamma)|=|\hat{1_{B'}}(\gamma)| /|B'|$ is at least $9/10N$ on the characters $\gamma\in S$ we have
$$\sum_x |f * \mu_{B'}(x)|^2 = N\sum_{\gamma} |\hat{f}(\gamma)|^2 |\hat{\mu}_B(\gamma)|^2 \geq \frac{N}{2}\sum_{\gamma \in S} |\hat{f}(\gamma)|^2\geq \frac{\alpha^2 \beta MN}{64}.$$
We now use this fact to extract a density increment on some $x$, that is large $f * \mu_{B'}(x)$. Notice that we are interested only in positive deviations from the mean so we cannot simply take average value, instead we use the fact that the mean of $f*\mu_{B'}(x)$ is 0 to write
$$\frac{\alpha^2 \beta M}{64}\leq \mathbb{E}_x (f * \mu_{B'}(x))^2 + \alpha\cdot \mathbb{E}_x \Bigl(f * \mu_{B'}(x)\Bigr) = \mathbb{E}_x (f * \mu_{B'}(x))(f * \mu_{B'}(x)+\alpha).$$
Now the second factor is always positive and so
$$\frac{\alpha^2\beta M}{64}\leq \max_x \Bigl(f * \mu_{B'}(x)\Bigr) \cdot \alpha.$$
Therefore there exists an $x$ for which
$$\frac{\alpha \beta M}{64} \leq \max_x \Bigl(1_R * \mu_{B'}(x)-\alpha\Bigr).$$
Therefore there exists a translate $x$ with 
$$\frac{|(x+B')\cap R|}{|B'|}\geq \alpha\Bigr(1+\frac{\beta M}{64}\Bigl).$$

We now show $B'$ actually contains a bohr set with a smaller set of generators. Let $C=\mathbb{Z}_n\setminus R$ and apply Chang's lemma to the set
$$\Lambda = \Bigl\{\gamma : |\hat{1_C}(\gamma)|\geq \frac{\eta}{1-\alpha}|C|/N\Bigr\}.$$
Observe that $S$ and $\Lambda$ are identical with the possible exception of the trivial character because $\hat{1_C}(\gamma)=-\hat{1_R}(\gamma)$ for non-trivial $\gamma$. Thus there exists a set of characters $\Gamma$ with
$$|\Gamma|\ll \Bigl(\frac{1-\alpha}{\eta}\Bigr)^2\log(N/|C|)\ll \Bigl(\frac{1-\alpha}{\eta}\Bigr)^2\log(2/(1-\alpha))$$
such that $S\subseteq\mathrm{spec}(\Gamma)$. As a consequence there exists a regular Bohr set $$B=\mathrm{Bohr}(\Gamma,\rho)\subseteq \mathrm{Bohr}(S,1/200) \subseteq B',$$
where $\rho = \Omega(|\Gamma|^{-2})$ such that $|B'_{1+\rho}|\leq 1.01|B'|$. By considering all translates of $B$ that intersect $B'$ we find one where the density increment is at least $1+\Omega(\beta M)$.
By Lemma~\ref{lem:find_progression} we extract an arithmetic progression of length at least $N^{\Omega(1/|\Gamma|)}$ from $B$ with the increment on $R$ at least $1+\Omega(\beta M)$. 
\end{proof}
In the next lemma we make use of a corollary of the Bogolyubov-Ruzsa lemma. This version can be also attributed to Sanders~\cite{sanders-structure}. It follows from the existence of a bohr set in $2A-2A$ and a well-known Ruzsa's covering lemma.
\begin{theorem}\label{thm:freiman}
    Suppose that for some fixed constant $K$ a set $A\subseteq \mathbb{Z}_N$ satisfies $|A+A|\leq K|A|$. Then there is $d=O(\log^4 K)$ and regular bohr set $B$ with $\mathrm{rank}(B)\leq d$ and $|B|\geq |A|/e^d$ such that $|A\cap B|\geq \exp(-C\log^4 K)|B|.$
\end{theorem}
Another ingredient is well-known Balog-Szemerédi-Gowers Theorem~\cite{balog-szemeredi, gowers1998}. Standard formulation is for the set $A-A$ but we may pass to $A+A$ by a simple application of Plünnecke's inequality and changing the constant from $5$ to $10$.
\begin{theorem}{(Balog-Szemerédi-Gowers)}\label{thm:bsg}
    Let $A\subseteq \mathbb{Z}$ and $E(A)\geq |A|^3/K$. Then there exists a set $A'\subseteq A$ such that $|A'|\gg |A|/K$ and 
    $$|A'+A'|\ll K^{10}|A'|.$$
\end{theorem}
\begin{lemma} \label{lem:large_energy}
Let $N,k\geq 1$ be integers and let $R\subseteq \mathbb{Z}_N$ be a set with $|R|=\alpha N$ where $\alpha\geq 1/10$.  Let $A_1,A_2,\cdots,A_k\subseteq \mathbb{Z}_N$ be disjoint sets with $|A_i|\geq \beta N$. Assume that $E(A_i)\geq  |A_i|^3/M$ and that if $x,y\in A_i$ for some $A_i$, then $x+y\notin R$.

Then there exists an arithmetic progression $P$ of length at least $N^{1/\Omega(k\log^4M)}$ such that we have
$$|P\cap R|/|P|\geq \Bigl(1+\frac{\beta k}{\exp(C\log^4 M)}\Bigr)\frac{|R|}{N}.$$
\end{lemma}
\begin{proof}
Apply Theorem~\ref{thm:bsg} to the sets $A_1,A_2,\cdots,A_k$ to find subsets $A_1',A_2',\cdots, A_k'$. To these subsets apply Theorem~\ref{thm:freiman} in order to obtain regular bohr sets $B_1,B_2,\cdots, B_k$ with the following properties.
$$\mathrm{rank}(B_i)\ll \log^4 M$$
$$\exp(\log^4 M)|A_i\cap B_i|\geq |A_i|.$$
Let $B'$ be a regular bohr set generated by the union of generators of $B_1,B_2,\cdots B_k$ and with width $\rho \approx \frac{1}{(k\log^4
M)^2}$  such that for every $B_i$ we have $|B_i^{1+\rho}|\leq 1.01|B_i|$.

Let us now use Lemma~\ref{lem:find_progression} to find an arithmetic progression $P$ of length at least $N^{1/\Omega(k\log^4M)}$ contained in $B'_{\epsilon}$ with $\epsilon\approx \rho^2$. Notice that $B'_{1-\epsilon} + P\subseteq B'$.

Define $T_i$ to be the set of all translates of $P$ such that the colour $A_i$ for $i\leq k$ intersects it in at least $|P|/8\exp(\log^4M)$ elements. To lower bound $|T_i|$ we write
$$|T_i||P| + (|B_i|-|T_i|)\frac{|P|}{8\exp(\log^4M)}\geq \frac{1}{2}\sum_{t\in B_i} |A_i\cap(t+P)|\geq \frac{1}{4}|A_i'||P|\geq \frac{|B_i||P|}{4\exp(\log^4M)}$$
and so we have $|T_i|\gg \frac{|B_i|}{\exp(\log^4M)}$. Let $T=\bigcup_i T_i$, then any $t\in T$ can belong to at most $\exp(\log^4 M)$ sets $T_i$ because total density of such colours can be at most $1$ within $t+P$. Therefore we have
$$|T|\geq \frac{\sum_{i=1}^k|T_i|}{\exp(\log^4M)}\geq\frac{\sum_{i=1}^k|B_i|}{\exp(\log^4M)^2}\geq\frac{\sum_{i=1}^k|A_i|}{\exp(\log^4M)^3}\geq \frac{\beta kN}{\exp(4\log^4M)}.$$

We now consider two options. Let us first suppose that for some $t\in T$ we have 
$$|R\cap (t+P)|\geq \alpha \Bigl(1-\frac{1/1000}{\exp(\log^4 M)}\Bigr)|P|.$$

Then we use Corollary~\ref{cor:progressions} with $\beta'=\frac{1}{8\exp(\log^4M)}$ to show a large, $1+\beta'/2$, density increment of $R$ on $t+P'$, where $|P'|\geq |P|^{1/4}$.

Otherwise we calculate average size of $|R\cap(t+P)|/|P|$ where $t\notin T$, call it $r$. Notice that
$$(N-|T|)r + |T|\Bigl(1-\frac{1/1000}{\exp(\log^4 M)}\Bigr)\alpha \geq \alpha N$$
and so
$$(N-|T|)r \geq \alpha(N-|T|) + \alpha\frac{|T|/1000}{\exp(\log^4 M)}.$$
Upon dividing both sides by $N-|T|$ we have a density increment on some $x\notin T$ of
$$1+\frac{|T|/1000}{(N-|T|)\exp(\log^4 M)}\geq 1 + \frac{\beta k}{1000\exp(5\log^4 M)}.$$
\end{proof}
By combining both approaches we get a stronger version of the density increment lemma. Notice that we balance out Lemma~\ref{lem:small_energy} and Lemma~\ref{lem:large_energy} by choosing $M$ which gives us the same density increment in both cases. We leave out sizes of progressions out of our optimisation as that would complicate the calculations with negligible gain.
\begin{lemma}\label{lem:general_increment}

Let $N,k\geq 1$ be integers and let $R\subseteq [N]$ be a set with $|R|=\alpha N$, where $\alpha\geq 1/10$.  Let $A_1,A_2,\cdots,A_k\subseteq [N]$ be disjoint subsets of $[N]$ with size at least $(1-\alpha)N/4k$. Assume that if $x,y\in A_i$ for some $A_i$, then $(x+y)/2\notin R$.

Then there exists an arithmetic progression $P$ of length at least $N^{ck^{-2}\log(2/(1-\alpha)))^{-1}}$ such that we have
$$|P\cap R|/|P|\geq \Bigl(1+\frac{1-\alpha}{k}\exp(c\log^{1/4} k)\Bigr)\frac{|R|}{N}.$$
\end{lemma}
\begin{proof}
We begin by moving the problem to $\mathbb{Z}_N$ without losing generality. Of course we have to do it in a way that does not introduce spurious solutions so let us proceed carefully. We start by splitting each $A_i$ into even and odd elements.
$$A_i=A_i^{odd}\cup A_i^{even}.$$
One of the sets $A_i^{odd}$, $A_i^{even}$ has at least $|A_i|/2$ elements. Then choose our new $A_i$ to be 
$$\{(a-1)/2 : a\in A_i^{odd}\},$$
or
$$\{a/2 : a\in A_i^{even}\},$$
whichever is larger. Let $I_1$ be the indices $i$ for which the first alternative was chosen and $I_2$ for which the second. Choose $I_1$ if $\sum_{i\in I_1}|A_i| > \sum_{i\in I_2}|A_i|$ and $I_2$ otherwise. Remove all remaining indices. If we have chosen $I_2$, let us keep $R$ as it is, if we have chosen $I_1$ let us subtract 1 from every element of $R$. This way we have sets $A_i$ for $i\in I_1$ or $i\in I_2$ such that all their elements are smaller than $N/2$ and so no wrap-around occurs. Thus for all such $A_i$ we have
$$\sum_{x\in \mathbb{Z}_N}1_R(x)1_{A_i}*1_{A_i}(x)=0.$$
We combine Lemma~\ref{lem:small_energy} and Lemma~\ref{lem:large_energy} with $M=\exp(C'\log^{1/4}k)$. 
Now if for any $A_i$ we have $$E(A_i)\leq |A_i|^3/M$$ we are in position to use Lemma~\ref{lem:small_energy}. If that is not the case we use Lemma~\ref{lem:large_energy}.
In one case our density increment is $1+\Omega((1-\alpha) M/k)$ and in the other case $1+\frac{1-\alpha}{\exp(C\log^{4} M)}$. Since for our choice of $M$ these two are optimised to be larger than $1+(1-\alpha)\exp(c\log^{1/4} k)/k$ we have finished the proof. 
\end{proof}
\begin{theorem}\label{thm:longer_progression_in_integers}
Let $k,N$ be positive integers with $k<N$ and let
$$A_1\cup A_2\cup\cdots\cup A_k\cup R=[N]$$ 
be a partitioning. Assume that if $x,y\in A_i$ for some $A_i$, then $(x+y)/2\notin R$ and that $|R|\geq N/10.$ Then $R$ contains an arithmetic progression of length at least $$N^{\exp(-O(k\log\log(2k))}$$    
\end{theorem}
\begin{proof}
The reasoning is similar to the proof of Theorem~\ref{thm:progression_in_integers} but we get stronger density increment on a shorter progression, thus we need to calculate how these two changes balance out. We iterate Lemma~\ref{lem:general_increment} $s$ times ($s$ to be decided later). Let $N_i$ be the length of arithmetic progression on which $R$ is dense after $i$ steps, so that $N_0=N$. Similarly as before let $P_i$ be the consecutive progressions and let $\alpha_i$ be the density of $R$ on $P_i$. Let $\beta_i=(1-\alpha_i)/4k$. Since this is $1/4$ of average colour density on $P_i$ all colours below this density can cover at most the fraction of $1/4$ of $P_i$. It is not necessarily the case that there are $k/4$ colours of density at least $\beta_i$ but we can assume so without loss of generality for the following reason. Choose the largest colour in $P_i$ and from any $\beta_i N_i$ elements of it create a brand new colour, which obviously satisfies the assumption. After at most $k$ steps we removed at most $\beta_i k = (1-\alpha)/4$ elements. Thus there must exist a colour with at least $\beta_i P_i$ elements and we can apply Lemma~\ref{lem:general_increment} with the same $k$. Similarly as before
$$\alpha_{i+1}\geq \Bigl(1+\frac{(1-\alpha_i)\exp(c\log^{1/4}k)}{k}\Bigr)\alpha_i$$
and so
$$\frac{1-\alpha_{i+1}}{1-\alpha_i} \leq 1-\frac{\alpha_i \exp(c\log^{1/4}k)}{k}\leq 1-\frac{\exp(c\log^{1/4}k)}{10k}.$$
Thus after at most $20k/\exp(c\log^{1/4}k)$ steps the size of the complement of $R$ halves, that is for any $i$ we have 
\begin{equation}
\frac{1-\alpha_{i+20k/\exp(c\log^{1/4}k)}}{1-\alpha_i} \leq \frac{1}{2}. \tag{*}
\end{equation}

We can apply Lemma~\ref{lem:termination} once $1-\alpha_i\leq (10k)^{-6}$ and so we can assume the opposite for all $i\leq s$. Moreover, by (*) we see that taking $s=O(k/\exp(c\log^{1/4}k))$ suffices (the $\log(2k)$ term is overpowered by the $\exp(c\log^{1/4}k)$ term). By the conclusion of Lemma~\ref{lem:general_increment} we have
$N_{i+1}\geq N_i^{ck^{-2}\log(2/(1-\alpha_i)))^{-1}}\geq N_i^{ck^{-2}\log (2k)}$.

Thus we apply Lemma~\ref{lem:termination} to the interval of length
$$N_s\geq N_i^{(ck^{-2}\log (2k))^s} \geq N^{\exp(-O(k/\exp(c\log^{1/4}k)))}$$
thus getting the desired length of a progression in $R$ owed to the worse progression length coming from Lemma~\ref{lem:termination}.
\end{proof}
\section{Use of AI in the proofs}
All arguments are human-generated, with the exception of Lemma~4 and the corresponding termination argument in the proof of Theorem~\ref{thm:f3_subspace}, which were suggested by ChatGPT. Previous argument used the pigeonhole principle once density of $R$ was sufficient, which resulted in more iterations. The whole document is written by the author.

\textsc{Faculty of Mathematics and Computer Science, Adam Mickiewicz University, Umultowska 87, 61-614 Poznan, Poland }

\textit{Email address:} tomasz.kosciuszko@amu.edu.pl
\end{document}